\documentclass[12pt,a4paper,leqno]{amsart}

\title[Microlocal properties of pseudodifferential operators]{Microlocal properties of Shubin pseudodifferential and localization operators}

\author[R. Schulz]{Ren\'e Schulz}

\author[P. Wahlberg]{Patrik Wahlberg}

\usepackage{tikz,xifthen}
\usepackage{latexsym}
\usepackage{amsmath}
\usepackage{amssymb}
\usepackage{amsthm}
\usepackage{amsfonts}
\usepackage{mathrsfs}
\usepackage{calc}
\usepackage{cite}
\usepackage{color}
\usepackage{graphicx}
\usepackage{enumerate}

\numberwithin{equation}{section}          %Detta gr att man f•r
\newtheorem{thm}{Theorem}
\numberwithin{thm}{section}

\newcommand{\rubrik}{}
\newtheorem{prop}[thm]{Proposition}
\newtheorem{cor}[thm]{Corollary}
\newtheorem{lem}[thm]{Lemma}

\theoremstyle{definition}

\newtheorem{defn}[thm]{Definition}
\newtheorem{example}[thm]{Example}

\theoremstyle{remark}

\newtheorem{rem}[thm]{Remark}              %T o m hit r bara allmn
\newcommand{\pd}[1] {\partial ^#1}
\newcommand{\pdd}[2] {\partial_{#1} ^{#2}}

\newcommand{\ro}{\mathbb R}
\newcommand{\no}{\mathbb N}
\newcommand{\rr}[1]{\mathbb R^{#1}}
\newcommand{\nn}[1]{\mathbb N^{#1}}

\newcommand{\co}{\mathbb C}

\newcommand{\dd}{\mathrm {d}}

\newcommand{\ep}{\varepsilon}
\newcommand{\fy}{\varphi}

\newcommand{\supp}{\operatorname{supp}}

\newcommand{\wpr}{{\text{\footnotesize $\#$}}}

\newcommand{\eabs}[1]{\langle #1\rangle}

\newcommand{\Sp}{\operatorname{Sp}}

\newcommand{\charac}{\operatorname{char}}
\newcommand{\conesupp}{\operatorname{conesupp}}
\newcommand{\dbar}{{{{\ \mathchar'26\mkern-12mu \mathrm d}}}}

\newcommand{\WF}{\mathrm{WF}}

\newcommand{\cS}{\mathscr{S}}

\newcommand{\cD}{\mathscr{D}}

\newcommand{\wh}{\widehat}

\def\la{\langle}
\def\ra{\rangle}

\begin{document}

\begin{abstract}
We investigate global microlocal properties of localization operators and Shubin pseudodifferential operators. The microlocal regularity is measured in terms of a scale of Shubin-type Sobolev spaces. In particular, we prove microlocality and microellipticity of these operators.
\end{abstract}

\keywords{Gabor wave front set, Shubin calculus, localization operator, microlocality, microellipticity}
\subjclass[2010]{35S05, 35A18, 35A22, 35A27, 42B37}

\maketitle

{\let\thefootnote\relax\footnotetext{Corresponding author: Patrik Wahlberg,
phone +46470708637. Authors' addresses. 
R. Schulz: Leibniz Universit\"at Hannover, Institut f\"ur Analysis, Welfenplatz 1, D--30167 Hannover, Germany, rschulz[AT]math.uni-hannover.de, 
P. Wahlberg: Department of Mathematics, Linn\ae us University, 351 95 V\"axj\"o, Sweden, patrik.wahlberg[AT]lnu.se.}}

%%%%%%%%%%%%%%%%%%%%%%%
\section{Introduction}
%%%%%%%%%%%%%%%%%%%%%%%

The microlocal approach to the study of singularities of a tempered distribution, that is in terms of some wave front set, may be viewed as the study of a resolution of singularities in the phase space. 
Several types of wave front sets have been investigated, including those that also encode growth singularities apart from singularities defined by lack of derivatives. 

One of these wave front sets was introduced in \cite{Hormander0} and further investigated in \cite{Rodino1}, therein called the Gabor wave front set $\WF_G(u)$ of a tempered distribution $u$. 
It consists of phase space directions in which a tempered distribution deviates from being both smooth and rapidly decaying, i.e. a Schwartz function. 
It has been shown that the Gabor wave front set may be characterized by the lack of rapid decay of the Gabor (or short-time Fourier) transform in open cones of the phase space. 
A main application of the Gabor wave front set concerns propagation of singularities for Schr\"odinger-type equations, see e.g. \cite{Cordero2,Nakamura1,Nicola2,Rodino2,SW}.  

A natural class of operators associated to the short-time Fourier transform is localization operators (also called anti-Wick-quantized operators or Toeplitz operators), see e.g. \cite{Cordero1,Folland1}. These may be interpreted as time-frequency multiplication operators defined by a symbol. 

It is a natural question to study the interplay of the Gabor wave front set and localization operators, and this paper is intended to contribute to  this field. Many of the arising questions may be transferred to questions regarding pseudodifferential operators. This is usually achieved by exploiting the Weyl--Wick connection, which states that any localization operator may be written as a Weyl-quantized pseudodifferential operator with a symbol that is the convolution of the anti-Wick symbol with the Wigner distribution of the window function. 

The Shubin calculus of pseudodifferential operators has been used to prove results on microlocality and microellipticity of pseudodifferential operators with respect to the Gabor wave front set \cite{Hormander0,Rodino1}. Roughly speaking these state that the application of a time-frequency filter to a given signal does not create more singularities than originally present, and filters out at most singularities situated at points where its symbol vanishes. 

In this paper we study a finer resolution of singularities by replacing the Gabor wave front set with the notion $\WF_{Q^s} (u) \subseteq \WF_G(u)$, called the Sobolev--Gabor wave front set of order $s \in \ro$, based on the scale of Shubin-type Sobolev spaces $Q^s$. 
The rapid decay in open cones in phase space, that negatively defines the Gabor wave front set, is then replaced by polynomially weighted square integrability in open cones. 
The Sobolev--Gabor wave front set was introduced by Nicola and Rodino \cite{Nicola2}. It is in fact a special case of a general modulation space construction where the exponent $p=2$ is relaxed to $1 \leq p \leq \infty$ \cite{Cordero2}. 

While the Sobolev--Gabor wave front set has been used in \cite{Nicola2} to study propagation of singularities for Schr\"odinger equations, few of its other properties have been investigated and the present paper is intended to investigate some of its basic features. Our main results concern microlocality and microellipticity of Shubin pseudodifferential operators with respect to Sobolev--Gabor wave front sets, of orders that differ by the order of the operator. Together with  
\begin{equation*}
\WF_G(u) = \overline{\bigcup_{s \in \ro} \WF_{Q^s}(u)} \subseteq T^* \rr d \setminus 0,
\end{equation*}
(see Proposition \ref{WFequality})
we recover known microlocal and microelliptic results for the Gabor wave front set. 

The paper is structured as follows.
In Section \ref{sec:prelim} we introduce notation and preliminary results. In Section \ref{sec:wfintro} the Sobolev--Gabor wave front set is described and related to the Gabor wave front set, and several characterizations are investigated. Finally, in Section \ref{sec:micro} we prove microlocality and microellipticity with respect to the Sobolev--Gabor wave front set for pseudodifferential and localization operators.

%%%%%%%%%%%%%%%%%%%%
\section{Preliminaries}\label{sec:prelim}
%%%%%%%%%%%%%%%%%%%%

An open ball of radius $r>0$ and center at the origin in $\rr d$ is denoted $B_r$,
and the unit sphere in $\rr d$ is denoted $S_{d-1}$.
We write $f (x) \lesssim g (x)$ provided there exists $C>0$ such that $f (x) \leq C g(x)$ for all $x$ in the domain of $f$ and $g$. The Japanese bracket on $\rr d$ is defined by $\la x \ra = \sqrt{1+|x|^2}$. For it, Peetre's inequality $\la{x+y}\ra^{s}\lesssim \la{x}\ra^{s}\la{y}\ra^{|s|}$, $s \in\ro$, is valid.

The Fourier transform on the Schwartz space $\mathscr S(\rr d)$ is normalized as
\begin{equation*}
\mathscr{F} f(\xi) = \wh f(\xi) = \int_{\rr d} f(x) e^{- i x \cdot \xi} \dd x, \quad \xi \in \rr d, \quad f \in \mathscr S(\rr d),
\end{equation*}
and extended to its dual, the tempered distributions $\mathscr S'(\rr d)$. 

The inner product $(\cdot,\cdot)$ on $L^2 (\rr d) \times L^2(\rr d)$ is conjugate linear in the second argument. 
We also use these notations to denote the (conjugate) linear action of $\cS'(\rr d)$ on $\cS(\rr d)$.

In the following we recall some notions of time-frequency analysis and pseudodifferential operators on $\rr d$. For these topics \cite{Folland1,Grochenig1,Hormander0,Lerner1,Nicola1,Shubin1} may serve as general references.

By $\psi_0\in\cS(\rr{d})$ we denote the $L^2$-normalized standard Gaussian window function $\psi_0(y) = \pi^{-d/4}e^{-\frac{1}{2}|y|^2}$.
Let $u \in \cS'(\rr d)$ and $\psi \in \cS(\rr d) \setminus 0$. The \emph{short-time Fourier transform} (STFT)\footnote{The STFT is also called Gabor transform in case $\psi=\psi_0$ and is closely connected to the so-called Bargmann and Fourier--Bros--Iagolnitzer transforms.} $V_\psi u $ of $u$ with respect to the window function $\psi$, is defined as
\begin{equation*}
V_\psi u :\ \rr {2d} \mapsto \co, \quad z \mapsto V_\psi u(x,\xi) = (u, \Pi(z) \psi ),
\end{equation*}
where $\Pi(z)=M_\xi T_x$, $z=(x,\xi) \in \rr {2d}$, is the time-frequency shift composed of the translation operator $T_x\psi(y)=\psi(y-x)$ and the modulation operator $M_\xi \psi(y)=e^{i y \cdot \xi} \psi(y)$.

We have $V_\psi u \in C^\infty(\rr {2d})$ and there exists $N \in \no$ such that 
\begin{equation*}
|V_\psi u(z)| \lesssim \eabs{z}^N, \quad z \in \rr {2d}.
\end{equation*}
Let $\psi \in \cS(\rr d)$ satisfy $\| \psi \|_{L^2}=1$.
The Moyal identity
\begin{equation*}
(u,g)= (2\pi)^{-d} \int_{\rr {2d}} V_\psi u(z) \, \overline{V_\psi g(z)} \, \dd z, \quad g \in \cS(\rr d), \quad u \in \cS'(\rr d),
\end{equation*}
is sometimes written
\begin{equation*}
u = (2\pi)^{-d} \int_{\rr {2d}} V_\psi u(x,\xi) \, M_\xi T_x \psi \, \dd x \, \dd \xi, \quad u \in \cS'(\rr d),
\end{equation*}
with action understood to take place under the integral. In this form it is an inversion formula for the STFT.
It can also be written $(2\pi)^{-d} V_\psi^* V_\psi = I$ which is formulated with the adjoint 
\begin{equation*}
V_\psi^* F = \int_{\rr {2d}} F(z) \, \Pi(z) \psi \, \dd z 
\end{equation*}
that satisfies 
\begin{equation*}
(V_\psi^* F,g)_{L^2(\rr d)}= (F,V_\psi g )_{L^2(\rr {2d})}, 
\end{equation*}
for $g \in \cS(\rr d)$ and $F \in (L_{\rm loc}^\infty \cap \cS')(\rr {2d})$, and extends to $F \in \cS' (\rr {2d})$. 

Let $a \in C^\infty (\rr {2d})$ and $m \in \ro$. Then $a$ is a \emph{Shubin symbol} of order $m$, denoted $a\in G^m$, if for all $\alpha,\beta \in \nn d$ there exists a constant $C_{\alpha\beta}>0$ such that
\begin{equation}\label{eq:shubinineq}
|\partial_x^\alpha \partial_\xi^\beta a(x,\xi)|\leq C_{\alpha\beta}\langle (x,\xi)\rangle^{m-|\alpha|-|\beta|}, \quad x,\xi \in \rr d.
\end{equation}
The Shubin symbols form a Fr\'echet space where the seminorms are given by the smallest possible constants in \eqref{eq:shubinineq}.
We denote
\begin{equation*}
G^\infty = \bigcup_{m \in \ro} G^m
\end{equation*}
and obviously
\begin{equation*}
\bigcap_{m \in \ro} G^m = \cS (\rr {2d}). 
\end{equation*}

For $a \in G^m$ a Weyl-quantized pseudodifferential operator is defined by
\begin{equation*}
a^w(x,D) u(x)
= \int_{\rr {2d}} e^{i(x-y) \cdot \xi} a\big((x+y)/2,\xi\big) \, u(y) \, \dd y \, \dbar \xi, \quad u \in \cS(\rr d),
\end{equation*}
when $m<-d$, where we use the convention $\dbar \xi=(2\pi)^{-d} \dd \xi$. The definition extends to general $m \in \ro$ if the integral is viewed as an oscillatory integral.
The operator $a^w(x,D)$ then acts continuously on $\cS(\rr d)$ and extends by duality to a continuous operator on $\cS'(\rr d)$. We remark that this quantization procedure may be extended to a weak formulation which yields continuous  operators $a^w(x,D):\cS(\rr{d}) \mapsto \cS'(\rr d)$, even if $a$ is only an element of $\cS'(\rr {2d})$.

Let $(a_j)_{j \geq 0}$ be a sequence of symbols such that $a_j \in G^{m_j}$ and $m_j \rightarrow - \infty$ as $j \rightarrow +\infty$, and set $m=\max_{j \geq 0} m_j$.
Then there exists a symbol $a \in G^{m}$, unique modulo $\cS(\rr {2d})$, such that
\begin{equation*}
a - \sum_{j=0}^{n-1} a_j \in G^{m_n'}, \quad n \geq 1, \quad m_n' = \max_{j \geq n} \, m_j. 
\end{equation*}
This is called an asymptotic expansion and denoted $a \sim \sum_{j \geq 0} a_j$.

The Weyl calculus enjoys the property $a^w(x,D)^* = \overline{a}^w(x,D)$ where $a^w(x,D)^*$ denotes the formal adjoint.
The Weyl product is the symbol product corresponding to composition of operators, $(a \wpr b)^w(x,D) = a^w(x,D) \, b^w (x,D)$.
It is a bilinear continuous map
\begin{equation*}
\wpr: \, G^m \times G^n \mapsto G^{m+n}, \quad m,n \in \ro.
\end{equation*}
We have the following asymptotic expansion for the Weyl product of $a \in G^m$ and $b \in G^n$, $m,n \in \ro$:
\begin{equation}\label{calculuscomposition1}
a \wpr b(x,\xi) \sim \sum_{\alpha, \beta \geq 0} \frac{(-1)^{|\beta|}}{\alpha! \beta!} \ 2^{-|\alpha+\beta|}
D_x^\beta \pdd \xi \alpha a(x,\xi) \, D_x^\alpha \pdd \xi \beta b(x,\xi), 
\end{equation}
where $D_j=-i \partial_j$, $1 \leq j \leq d$.

The conical support $\conesupp (a)$ of $a \in \cD'(\rr n)$ is defined as the set of all
$x \in \rr n \setminus 0$ such that any conic open set $\Gamma_x \subseteq \rr n \setminus 0$ containing $x$ satisfies:
\begin{equation*}
\overline{\supp (a) \cap \Gamma_x} \quad \mbox{is not compact in} \quad \rr n.
\end{equation*}

We define the microsupport $\mu \supp (a)$ of a symbol $a \in G^m$ analogously to its definition for H\"ormander symbols \cite[p.~118]{Folland1}.\footnote{This notion is also called the wave front set of $a^w(x,D)$ \cite[Chapter~18.1]{Hormander0} and the essential support of $a^w(x,D)$ \cite{Taylor1}.}
If $a \in G^m$ and $z_0 \in T^* \rr d \setminus 0$ then $z_0 \notin \mu \supp (a)$ provided there exists a conic open set $\Gamma \subseteq T^* \rr d \setminus 0$ containing $z_0$ such that 
\begin{equation}\label{musuppdef}
\sup_{z \in \Gamma} \eabs{z}^N \left| \partial^\alpha a(z) \right| < \infty, \quad \alpha \in \nn {2d}, \quad N \geq 0.
\end{equation}
Clearly we have for $a \in G^\infty$
\begin{equation*}
\mu \supp (a) \subseteq \conesupp (a). 
\end{equation*}

Let $a \in G^m$. As in \cite{Shubin1} we call $a$ hypoelliptic of order $m^\prime \leq m$ if it fulfills the estimates
\begin{equation}\label{noncharlowerbound2}
\begin{aligned}
|a(z )| &\geq C \eabs{z}^{m'}, \quad |z| \geq A, \\
|\partial^\alpha a(z)| &\leq C_\alpha |a(z)| \eabs{z}^{-|\alpha|}, \quad \alpha \in \nn {2d}, \quad |z| \geq A,
\end{aligned}
\end{equation}
for suitable $C, C_\alpha, A>0$. The space of all such symbols is denoted $HG^{m,m'}$. 
The symbols in $HG^{m,m}$ are called elliptic.

The concept of hypoellipticity is micro-localizable in the following sense:
\begin{defn}\label{noncharacteristic1}
Let $a\in G^m$. A point in the phase space $z_0 \in T^* \rr d \setminus 0$ is called non-hypercharacteristic of order $m' \leq m$ for $a$ provided there exists an open conic set $\Gamma \subseteq T^* \rr d \setminus 0$ such that $z_0 \in \Gamma$ and $a$ fulfils the estimates \eqref{noncharlowerbound2} when $z \in \Gamma$.
\end{defn}
For $a \in G^m$ and $m' \leq m$ we denote the hypercharacteristic set of $a$ by $\charac_{m^\prime}(a)$ and define it as the set of all $z \in T^* \rr d \setminus 0$ such that $z$ is not non-hypercharacteristic. Note that $\charac_{m''}(a) \subseteq \charac_{m'}(a)$ when $m'' \leq m'$. The special case $\charac_{m}(a) = \charac (a)$ is called the characteristic set and was defined in \cite{Hormander1}. 
\begin{rem}
Note that
\begin{equation*}
\mu \supp (a) \bigcup \left( \bigcap_{m' \leq m} \charac_{m'} (a) \right) = T^* \rr d \setminus 0, \quad a \in G^m,
\end{equation*}
and
$\charac_{m^\prime}(a)=\emptyset \Leftrightarrow a\in HG^{m,m^\prime}.$
\end{rem}

A class of operators related to pseudodifferential operators is that of localization operators, also called anti-Wick-quantized operators \cite{Cordero1,Shubin1}. 
For $a\in \cS'(\rr{2d})$ we define the corresponding localization operator $A_a$ weakly by its action on $f,g\in\cS(\rr{d})$ via 
\begin{equation*}
(A_a f,g) = (2\pi)^{-d}(a V_{\psi_0}f,V_{\psi_0}g) = (2\pi)^{-d} (V_{\psi_0}^* a V_{\psi_0}f,g), 
\end{equation*}
that is, $A_a = (2\pi)^{-d} V_{\psi_0}^* a V_{\psi_0}$.
The next result (cf. \cite[Section~1.7.2]{Nicola1}) says that $A_a$ can be written as a Weyl pseudodifferential operator with a symbol that is the convolution of $a$ and a Gaussian. 

\begin{prop}[the Weyl--Wick connection]\label{WW}
If $a\in \cS'(\rr{2d})$ then $A_a=b^w(x,D)$ where
\begin{equation}
\label{eq:weylwick}
b=\pi^{-d} e^{-|\cdot|^2}*a. 
\end{equation} 
If $a\in G^m$ then $b\in G^m$, 
\begin{equation}
\label{eq:weylwickexp}
b\sim\sum_{\alpha\in\nn{2d}} c_{\alpha} \partial^\alpha a 
\end{equation} 
with $c_{0}=1$, 
$\charac_{m'} (b) \subseteq \charac_{m'}(a)$ when $m' \leq m$, 
and $\mu\supp(b) \subseteq \mu\supp(a)$.
\end{prop}

\begin{proof}
The formulas \eqref{eq:weylwick} and \eqref{eq:weylwickexp} are proved in \cite[Proposition~1.7.9 and Theorem~1.7.10]{Nicola1}, respectively (cf. \cite[Theorem~24.1]{Shubin1}). If $a\in G^m$ then $b\in G^m$ follows from \eqref{eq:weylwick} and Peetre's inequality.

To prove the statement about the hypercharacteristic set, let $m' \leq m$ and suppose $0 \neq z_0 \notin \charac_{m'}(a)$. 
Then there exists a conic open set $\Gamma \subseteq T^* \rr d \setminus 0$ and $C,C_\alpha,A>0$ such that $z_0 \in \Gamma$ 
and the estimates \eqref{noncharlowerbound2} are valid when $z \in \Gamma$. 
From the asymptotic expansion \eqref{eq:weylwickexp}, picking $n \geq m-m'+1$, we have
\begin{equation}\label{bseries}
b = \sum_{0 \leq |\alpha| \leq n-1} c_\alpha \pd \alpha a + a_n
\end{equation} 
where $a_n \in G^{m-n} \subseteq G^{m'-1}$. 

Using \eqref{noncharlowerbound2} this gives if $z \in \Gamma$ and $|z| \geq A$
\begin{equation}\label{denominator1}
\begin{aligned}
|b(z)| & = \left| a(z) + \sum_{0 < |\alpha| \leq n-1} c_\alpha \pd \alpha a (z) + a_n(z) \right| \\
& \geq |a(z)| - \left( \sum_{0 < |\alpha| \leq n-1} |c_\alpha| \, | \pd \alpha a(z)| + |a_n(z)| \right) \\
& \geq |a(z)| \left( 1 - C_1 \eabs{z}^{-1}  - \frac{|a_n(z)|}{|a(z)|} \right) \\
& \geq |a(z)| \left( 1 - C_2 \eabs{z}^{-1} \right)
\end{aligned}
\end{equation} 
for constants $C_1,C_2>0$. Possibly augmenting $A>0$ we thus have the estimate
\begin{equation*}
|b(z)| \geq C_3 \eabs{z}^{m'}, \quad z \in \Gamma, \quad |z| \geq A, 
\end{equation*} 
for $C_3>0$, thus confirming the first of the two estimates corresponding to \eqref{noncharlowerbound2} for $b$ and $z \in \Gamma$.  

Concerning the second estimate that must be shown in order to prove that $z_0 \notin \charac_{m'}(b)$, let $\beta \in \nn {2d}$. 
From \eqref{bseries} we estimate, 
again using \eqref{noncharlowerbound2}, for $z \in \Gamma$ and $|z| \geq A$,
\begin{equation}\label{nominator1}
\begin{aligned}
|\pd \beta b(z)| & \leq \sum_{0 \leq |\alpha| \leq n-1} |c_\alpha| \, | \partial^{\alpha+\beta} a (z) | + |\pd \beta a_n(z)| \\
& \lesssim |a(z)| \eabs{z}^{-|\beta|} \left( 1 + \frac{\eabs{z}^{m-n} }{|a(z)|} \right) \\
& \lesssim |a(z)| \eabs{z}^{-|\beta|} \left( 1 + \eabs{z}^{m-m'-n}  \right) \\
& \lesssim |a(z)| \eabs{z}^{-|\beta|}. 
\end{aligned}
\end{equation} 

Combining \eqref{denominator1} and \eqref{nominator1} yields finally, again after possibly augmenting $A>0$, 
\begin{equation*}
\left| \frac{\pd \beta b(z)}{b(z)} \right| \lesssim \eabs{z}^{-|\beta|}, \quad z \in \Gamma, \quad |z| \geq A. 
\end{equation*} 
This verifies the second estimate of  \eqref{noncharlowerbound2} for $b$ and $z \in \Gamma$, and thus we have proved $z_0 \notin \charac_{m'}(b)$
which completes the proof of $\charac_{m'}(b) \subseteq \charac_{m'} (a)$. 

Finally, to prove $\mu\supp(b) \subseteq \mu\supp(a)$ we let $0 \neq z_0 \notin \mu\supp(a)$. 
Thus there exists a conic open set $\Gamma \subseteq T^* \rr d \setminus 0$ containing $z_0$ such that \eqref{musuppdef} holds. 
Let $\Gamma' \subseteq T^* \rr d \setminus 0$ be an open cone such that $z_0 \in \Gamma'$ and $\overline{\Gamma' \cap S_{2d-1}} \subseteq \Gamma$. 
Let $\alpha \in \nn {2d}$ and $\ep>0$. We split the convolution integral as
\begin{equation*}
\pi^d \pd \alpha b(z) = \underbrace{\int_{\eabs{w} \leq \ep \eabs{z}} \pd \alpha a(z-w) \, e^{-|w|^2} \, dw}_{:= I_1}
+ \underbrace{\int_{\eabs{w} > \ep \eabs{z}} \pd \alpha a(z-w) \, e^{-|w|^2} \, dw}_{:= I_2}. 
\end{equation*} 
If $z \in \Gamma'$, $|z| \geq 1$ and $\eabs{w} \leq \ep \eabs{z}$ then $z-w \in \Gamma$ if $\ep>0$ is chosen sufficiently small. 
Using \eqref{musuppdef} this gives for any $N \geq 0$ 
\begin{equation}\label{I1estimate}
\begin{aligned}
|I_1| & \lesssim \int_{\eabs{w} \leq \ep \eabs{z}} \eabs{z-w}^{-N} \, e^{-|w|^2} \, dw 
\lesssim \eabs{z}^{-N} \int_{\eabs{w} \leq \ep \eabs{z}} \eabs{w}^{N} \, e^{-|w|^2} \, dw \\
& \lesssim \eabs{z}^{-N}, \quad z \in \Gamma', \quad |z| \geq 1. 
\end{aligned}
\end{equation} 
The remaining integral we estimate for any $N \geq 0$ as
\begin{equation}\label{I2estimate}
\begin{aligned}
|I_2| & \lesssim \int_{\eabs{w} > \ep \eabs{z}} \eabs{z-w}^{m-|\alpha|} \, e^{-|w|^2} \, dw \\
& \lesssim \eabs{z}^{-N} \int_{\eabs{w} > \ep \eabs{z}} \eabs{z}^{|m| +N} \, \eabs{w}^{|m|} \, e^{-|w|^2} \, dw \\
& \lesssim \eabs{z}^{-N} \int_{\eabs{w} > \ep \eabs{z}} \eabs{w}^{2|m| +N} \, e^{-|w|^2} \, dw \\
& \lesssim \eabs{z}^{-N}, \quad z \in \rr {2d}. 
\end{aligned}
\end{equation} 
Combining \eqref{I1estimate} and \eqref{I2estimate} shows that $z_0 \notin \mu\supp(b)$, which proves $\mu\supp(b) \subseteq \mu\supp(a)$. 
\end{proof}
The \textit{Shubin--Sobolev spaces} $Q^s(\rr d)$, $s \in \ro$, were introduced by Shubin \cite{Shubin1} (cf. \cite{Nicola1}).
They can be defined as the modulation space $M_s^{2,2} (\rr d)$ \cite{Grochenig1} 
\begin{equation*}
Q^s (\rr d)  = \{u \in \cS'(\rr d): \, \eabs{ \cdot }^s V_\fy u \in L^2(\rr {2d}) \} 
\end{equation*}
where $\fy \in \cS(\rr d) \setminus 0$ is fixed, with norm
\begin{equation*}
\| u \|_{Q^s} = \left\| \eabs{ \cdot }^s V_\fy u \right\|_{L^2(\rr {2d})} = \| u \|_{M_s^{2,2}}. 
\end{equation*}
Another $\fy \in \cS(\rr d) \setminus 0$ gives an equivalent norm.
 
According to  \cite[Lemma~2.3]{Boggiatto1} an equivalent norm can be formulated in terms of localization operators. 
To wit, we have 
\begin{equation*}
Q^s (\rr d)  = \{u \in \cS'(\rr d): \, A_{\eabs{\cdot}^s}  u \in L^2(\rr d) \} 
\end{equation*}
with norm
\begin{equation*}
\| u \|_{Q^s} = \left\| A_{\eabs{\cdot}^s} u \right\|_{L^2(\rr d)}. 
\end{equation*}

As a third alternative $Q^s$ may be defined by 
\begin{equation*}
Q^s (\rr d) = \{ u \in \cS'(\rr{d}):  \,\,a^w(x,D)u \in L^2(\rr d) \}
\end{equation*}
for any symbol $a\in HG^{s,s}$. 
Different choices of such a symbol and an associated (left) parametrix $P$ for $a^w(x,D)$, meaning $Pa^w(x,D)=I+R$ with $R:\cS'(\rr d) \mapsto \cS(\rr d)$ a continuous operator (regularizer), yield equivalent norms $\|u\|_{Q^s,a,R}=\|a^w(x,D)u\|_{L^2}+\| Ru\|_{L^2}$.

We list some well-known properties of the Shubin--Sobolev spaces, see \cite{Nicola1,Shubin1}.
\begin{prop}[Properties of Shubin--Sobolev spaces]\hspace{1pt}
\label{prop:Sobolevprop}
\begin{enumerate}
\item If $a\in G^m$ then $a^w(x,D)$ is a continuous map $Q^s(\rr d) \mapsto Q^{s-m}(\rr d)$ for $s \in \ro$.
\item If $m>0$, the embedding $Q^{s} (\rr d) \subseteq Q^{s-m} (\rr d)$ is compact for $s \in \ro$.
\item We have
\begin{equation*}
\bigcup_{s \in \ro} Q^s (\rr d) =\cS' (\rr d) \quad \bigcap_{s \in \ro} Q^s (\rr d) =\cS (\rr d).
\end{equation*}
\end{enumerate}
\end{prop}

%%%%%%%%%%%%%%%%%%%%%%%%%%%%%%%
\section{The Sobolev--Gabor wave front set}\label{sec:wfintro}
%%%%%%%%%%%%%%%%%%%%%%%%%%%%%%%

The Gabor wave front set of $u \in \cS'(\rr d)$ is defined as (cf. \cite{Hormander1})
\begin{equation*}
\WF_G (u) = \bigcap_{a \in G^\infty: \, a^w(x,D) u \in \cS} \charac(a) \subseteq T^* \rr d \setminus 0. 
\end{equation*}
The following characterization was showed in \cite{Hormander1}. 
We have $0 \neq z_0 \notin \WF_G(u)$ if and only if there exists an open conic set $\Gamma \subseteq T^* \rr d \setminus 0$ such that $z_0 \in \Gamma$ and 
\begin{equation*}
\sup_{z \in \Gamma} \eabs{z}^N | V_{\psi_0} u(z)| < \infty, \quad N \geq 0. 
\end{equation*}
This means that $V_{\psi_0} u$ decays super-polynomially in $\Gamma$. 
The characterization is invariant under a change of window function from $\psi_0$ to any function $\psi \in \cS(\rr d) \setminus 0$ \cite{Rodino1}.

The Gabor wave front set differs from the usual $C^\infty$-wave front and its global counterpart, the scattering- or $\mathrm{SG}$-wave front set in several aspects. We refer the reader to \cite{Hormander0,Rodino1,Schulz1} for comparisons of these notions.
A list of important properties of the Gabor wave front set follows. 

\begin{enumerate}

\item If $u \in \cS'(\rr d)$ then $\WF_G(u) = \emptyset$ if and only if $u \in \cS (\rr d)$ \cite[Proposition 2.4]{Hormander1}.

\item If $u \in \mathscr S'(\rr d)$, $m \in \ro$ and $a \in G^m$ then
\begin{equation}\label{microlocalelliptic1}
\begin{aligned}
\WF_G( a^w(x,D) u) 
& \subseteq \WF_G(u) \cap \mu \supp (a) \\
& \subseteq \WF_G( a^w(x,D) u) \ \bigcup \ \charac (a). 
\end{aligned}
\end{equation}

\item We have 
\begin{equation}\label{symplecticinvarianceWFG}
\WF_G( \mu(\chi) u) = \chi \WF_G(u), \quad \chi \in \Sp(d, \ro), \quad u \in \cS'(\rr d), 
\end{equation}
where $\Sp(d,\ro)$ is the group of real symplectic matrices, and where 
$\chi \mapsto \mu(\chi)$ is the metaplectic representation \cite{Folland1}, that satisfies 
\begin{equation*}
\mu(\chi)^{-1} a^w(x,D) \, \mu(\chi) = (a \circ \chi)^w(x,D), \quad a \in \cS'(\rr {2d}). 
\end{equation*}

\end{enumerate}

These properties are proved in \cite{Hormander1,Rodino1}, except 
\begin{equation}\label{sharpening1}
\WF_G( a^w(x,D) u) \subseteq \mu \supp (a), \quad a \in G^m, \quad u \in \cS'(\rr d). 
\end{equation}
The inclusion 
\begin{equation*}
\WF_G( a^w(x,D) u) \subseteq \conesupp (a). 
\end{equation*}
is shown in \cite[Proposition~2.5]{Hormander1}. 

To prove the sharpening \eqref{sharpening1}, 
let $0 \neq z_0 \in T^* \rr d \setminus \mu \supp (a)$. Then $z_0 \in \Gamma$ where $\Gamma \subseteq T^* \rr d \setminus 0$ is an open cone and \eqref{musuppdef} holds. Let $b \in G^0$ satisfy $\supp (b) \subseteq \Gamma$ and $z_0 \notin \charac (b)$. 
Combining \eqref{calculuscomposition1} with \eqref{musuppdef} gives $b \wpr a \in \cS(\rr {2d})$. 
Hence $b^w(x,D) a^w(x,D) u \in \cS(\rr d)$ which shows that $z_0 \notin \WF_G (a^w(x,D) u)$. 
We have now proved \eqref{sharpening1}. 

Next we define the Sobolev--Gabor wave front set that is the main object of study in this paper. 
It was introduced by Nicola and Rodino \cite{Nicola2} and studied with respect to propagation of singularities for Schr\"odinger equations.

\begin{defn}
\label{def:wfgl}
If $u \in \cS'(\rr d)$ and $\fy \in \cS (\rr d) \setminus 0$ then the \emph{Sobolev--Gabor wave front set} $\WF_{Q^s}(u)$ of order $s \in \ro$ is defined as follows. For $z_0 \in T^* \rr d \setminus 0$ we have $z_0 \notin \WF_{Q^s}(u)$ if 
there exists an open cone $\Gamma\subseteq \rr {2d} \setminus 0$ containing $z_0$ such that 
$\eabs{\cdot}^s V_\varphi u$ restricted to $\Gamma$ belongs to $L^2(\rr {2d})$, i.e.
\begin{equation*}
\int_{\Gamma} \eabs{z}^{2s} | V_\varphi u(z)|^2 \, d z < \infty. 
\end{equation*}
\end{defn}
Obviously $\WF_{Q^s}(u) \subseteq \WF_G(u)$ is a closed conic subset of $T^* \rr d \setminus 0$, and for $t \geq s$ we have $\WF_{Q^s}(u) \subseteq \WF_{Q^t}(u)$.
Moreover we have $\WF_{Q^s}(u) = \emptyset$ if and only if $u \in Q^s(\rr d)$. 
Thus $\WF_{Q^s}(u)$ may be interpreted as the phase space directions in which a tempered distribution fails to belong to $Q^s$. 
This observation will be elaborated in Proposition \ref{characterizationWFQs}. 

\begin{prop}\label{WFinvariance}
If $u \in \cS'(\rr d)$ and $s \in \ro$ then $\WF_{Q^s}(u)$ does not depend on the window function $\fy \in \cS (\rr d) \setminus 0$.
\end{prop}
\begin{proof}
Suppose $\fy \in \cS (\rr d) \setminus 0$ and $0 \neq z_0 \notin \WF_{Q^s}(u)$ defined by means of $\fy \in \cS (\rr d)$. 
Then 
\begin{equation}\label{coneL2s0}
\int_{\Gamma} \eabs{z}^{2s} | V_\fy u(z)|^2 \, d z < \infty 
\end{equation}
for some open cone $\Gamma \subseteq T^* \rr d \setminus 0$ containing $z_0$. 
Let $\psi \in  \cS (\rr d) \setminus 0$ and let $\Gamma' \subseteq T^* \rr d \setminus 0$ be an open cone such that $z_0 \in \Gamma'$ and 
$\overline{\Gamma' \cap S_{2d-1}} \subseteq \Gamma$. 
We will show that 
\begin{equation}\label{coneL2s1}
\int_{\Gamma'} \eabs{z}^{2s} | V_\psi u(z)|^2 \, d z < \infty 
\end{equation}
which means that $\WF_{Q^s}(u)$ can be defined by $\psi$ as well as $\fy$, which is the claimed independence of window function. 

By \cite[Lemma 11.3.3]{Grochenig1} we have 
\begin{equation*}
| V_\psi u(z)| \leq (2 \pi)^{-d} \| \fy \|_{L^2}^{-2} |V_\fy u| * |V_\psi \fy| (z), \quad z \in \rr {2d}. 
\end{equation*}
Denoting $g := V_\psi \fy \in \cS(\rr {2d})$, this and Minkowski's inequality give
\begin{equation}\label{coneL2s2}
\begin{aligned}
& \int_{\Gamma'} \eabs{z}^{2s} | V_\psi u(z)|^2 \, d z 
\lesssim \int_{z \in \Gamma'} \left( \int_{w \in \rr {2d}} \eabs{z}^{s} \, | V_\fy u(z-w)| \, |g (w)| \, dw \right)^2 \, dz \\
& \leq \left( \int_{w \in \rr {2d}}  |g(w)| \left( \int_{z \in \Gamma'} \eabs{z}^{2s} \, | V_\fy u(z-w)|^2 \, dz \right)^{1/2} \, dw \right)^2.
\end{aligned}
\end{equation}
For $\ep>0$ we split the region of integration in the inner integral as 
\begin{align*}
\int_{z \in \Gamma'} \eabs{z}^{2s} \, | V_\fy u(z-w)|^2 \, dz 
& = I_1+I_2+I_3,
\end{align*}
where $I_1$ is the integral over $\Gamma' \cap B_1$, $I_2$ that over all $z \in \Gamma' \setminus B_1$ that satisfy 
$\eabs{w}\leq \ep \eabs{z}$
and $I_3$ is the integral over the remainder of $\Gamma'$.
We estimate $I_1$, $I_2$ and $I_3$ separately. 

Since for some $M \geq 0$ 
\begin{equation*}
| V_\fy u(z)| \lesssim \eabs{z}^M, \quad z \in \rr {2d}, 
\end{equation*}
Peetre's inequality allows us to estimate 
$I_1 \lesssim \eabs{w}^{2M}$. 

Concerning $I_2$, $z \in \Gamma'$, $|z| \geq 1$ and $\eabs{w} \leq \ep \eabs{z}$ together imply $z-w \in \Gamma$ provided $\ep>0$ is sufficiently small. 
Thus, again using Peetre's inequality, 
\begin{align*}
I_2 & = \int_{z \in \Gamma', \, |z| \geq 1, \, \eabs{w} \leq \ep \eabs{z} } \eabs{z}^{2s} \, | V_\fy u(z-w)|^2 \, dz \\
& \lesssim  \eabs{w}^{2|s|} \int_{z \in \Gamma', \, |z| \geq 1, \, \eabs{w} \leq \ep \eabs{z} } \eabs{z-w}^{2s} \, | V_\fy u(z-w)|^2 \, dz \\
& \leq  \eabs{w}^{2|s|} \int_{z \in \Gamma} \eabs{z}^{2s} \, | V_\fy u(z)|^2 \, dz \\
& \lesssim \eabs{w}^{2|s|},
\end{align*}
in the final step using \eqref{coneL2s0}. 

Finally we estimate $I_3$: 
\begin{align*}
I_3 & = \int_{z \in \Gamma', \, |z| \geq 1, \, \eabs{w} > \ep \eabs{z} } \eabs{z}^{2s} \, | V_\fy u(z-w)|^2 \, dz \\
& \lesssim  \int_{z \in \Gamma', \, |z| \geq 1, \, \eabs{w} > \ep \eabs{z} } \eabs{z}^{2s} \, \eabs{z-w}^{2M} \, dz \\
& \lesssim  \eabs{w}^{2M} \int_{z \in \Gamma', \, |z| \geq 1, \, \eabs{w} > \ep \eabs{z} } \eabs{z}^{2s+2M+ 2 d+1} \, \eabs{z}^{-2 d - 1}  \, dz \\
& \lesssim  \eabs{w}^{2M + 2s+2M+ 2 d+1} \int_{\rr {2d}}  \eabs{z}^{-2 d - 1}  \, dz \\
& \lesssim  \eabs{w}^{4M + 2s+ 2 d+1}.  
\end{align*}

Thus $I_1$, $I_2$ and $I_3$ are each bounded by some power of $\eabs{w}$, which inserted into \eqref{coneL2s2}, keeping in mind $g \in \cS(\rr {2d})$, shows that \eqref{coneL2s1} holds as claimed. 
\end{proof}

As a consequence of this result we note that the symplectic invariance 
\eqref{symplecticinvarianceWFG} for the Gabor wave front set extends to the Sobolev--Gabor wave front set of any order $s \in \ro$. 
In fact, by \cite[Eq.~(3.15)]{Wahlberg1} we have 
\begin{equation*}
\left| V_{\mu(\chi) \fy} \left( \mu(\chi) u \right)(\chi z) \right| = \left| V_\fy u (z) \right|
\end{equation*}
for $\fy \in \cS (\rr d)$, $u \in \cS'(\rr d)$, $\chi \in \Sp(d,\ro)$ and $z \in \rr {2d}$, 
and $\mu(\chi)$ is a continuous operator on $\cS(\rr d)$ (cf. \cite[Proposition~4.27]{Folland1}).

In the following we discuss how the Sobolev--Gabor wave front set may be characterized by means of pseudodifferential and localization operators. 
We first deduce a series of results that assure that the involved symbols may be modified as long as they fulfill certain support and hypoellipticity criteria around a given point of interest.

First we recall that if $m' \leq m$ and $a\in HG^{m,m^\prime}$, then there exists a parametrix $p\in HG^{-m^\prime,m}$ such that $p \wpr a-1\in\cS(\rr {2d})$ and $a \wpr p-1\in\cS(\rr {2d})$ (cf. \cite[Theorem~25.1]{Shubin1}). This result may be micro-localized in the following sense.

\begin{prop}\label{microelliptic}
Suppose $a \in G^m$ and $\charac_{m^\prime} (a) \neq T^* \rr d \setminus 0$ for some $m' \leq m$. 
Let $\Gamma \subseteq T^* \rr d \setminus 0$ be a closed conic set such that 
$\charac_{m'} (a) \cap \Gamma=\emptyset$. 
Then there exists $A>0$ such that for any 
$\chi \in G^0$ with $\supp(\chi) \subseteq \Gamma \setminus B_A$,
there exists $b \in G^{-m'}$ such that 
\begin{equation*}
b \wpr a = \chi + r 
\end{equation*}
where $r \in \cS(\rr {2d})$. 

Under the additional assumption 
$\Gamma \setminus \charac(\chi) \neq \emptyset$,
$b$ may be chosen to satisfy $\charac_{-m} (b) \cap \Gamma'=\emptyset$
for any open cone $\Gamma' \subseteq T^* \rr d \setminus 0$ such that $\overline{\Gamma' \cap S_{2d-1}} \subseteq \Gamma \setminus \charac(\chi)$. 
\end{prop}
\begin{proof}
The proof follows along the lines of that of \cite[Theorem~2.3.3]{Cordes1}, where the analysis is carried out for the  class of $\mathrm{SG}$-symbols that are micro-hypoelliptic in a certain sense.

Let us briefly recall the main steps of the parametrix construction. 
As a first approximation set $b_0:= a^{-1} \chi$. The estimates 
\begin{equation*}
|\partial^\alpha (a^{-1}) (z)| \leq C_\alpha |a(z)|^{-1} \eabs{z}^{-|\alpha|}, \quad \alpha \in \nn {2d}, \quad z \in \Gamma, \quad |z| \geq A, 
\end{equation*}
are consequences of the hypoellipticity estimates \eqref{noncharlowerbound2} and induction.   
They imply the estimates
\begin{equation*}
|\partial^\alpha b_0 (z)| \leq C_\alpha |a(z)|^{-1} \eabs{z}^{-|\alpha|}, \quad \alpha \in \nn {2d}, \quad |z| \geq A, 
\end{equation*}
and consequently $b_0 \in G^{-m'}$. 

Then, by \eqref{calculuscomposition1} and again the hypoellipticity estimates \eqref{noncharlowerbound2} it follows that $b_0\wpr a=\chi + r_0 + r_{0,\cS}$ with $r_0\in G^{-2}$ satisfying $\supp(r_0) \subseteq \supp(\chi)$ and $r_{0,\cS}\in\cS(\rr {2d})$. Subsequently, setting $b_1:=-a^{-1}r_0$, we notice that 
we obtain the estimates
\begin{equation*}
|\partial^\alpha b_1 (z)| \leq C_\alpha |a(z)|^{-1} \eabs{z}^{-2-|\alpha|}, \quad \alpha \in \nn {2d}, \quad |z| \geq A, 
\end{equation*}
and consequently $b_1\in G^{-m'-2}$. 
This gives
\begin{equation*}
(b_0+b_1) \wpr a = \chi + r_0 + r_{0,\cS} -r_0 + r_1 + r_{1,\cS} = \chi + r_1 + r_{0,\cS} + r_{1,\cS}
\end{equation*}
with $r_1\in G^{-4}$, $\supp(r_1) \subseteq \supp(\chi)$ and $r_{1,\cS}\in\cS(\rr {2d})$. Constructing in this way recursively  
$b_{j+1}:=-a^{-1} r_j \in G^{-m'-2(j+1)}$ and $r_{j+1} \in G^{-2(j+2)}$ with $\supp(r_{j+1}) \subseteq \supp(\chi)$, $j=1,2,\dots$, one obtains a sequence of symbols $(b_j)_{j \geq 0}$.
Finally set $b\sim \sum_{j=0}^\infty b_j\in G^{-m^\prime}$. The symbol $b$ satisfies $b \wpr a = \chi + r$ with $r \in \cS(\rr {2d})$.
This concludes the first part of the proof. 

It remains to verify the second claim, under the additional assumption $\Gamma \setminus \charac(\chi) \neq \emptyset$. 
This assumption implies $|\chi(z)| \geq \ep > 0$ when $z \in \Gamma' \setminus B_{A'}$ for some $A'>A$ 
for any open cone $\Gamma' \subseteq T^* \rr d \setminus 0$ such that $\overline{\Gamma' \cap S_{2d-1}} \subseteq \Gamma \setminus \charac(\chi)$. 

We note the following properties that hold due to the construction of $b_j \in G^{-m^\prime-2j}$, $j=0,1,2,\dots$ in the first part of the proof, 
together with the new assumptions on $\chi$. 
We have
\begin{equation}\label{b0lowerbound}
| b_0 (z)| \geq \ep  |a(z)|^{-1}, \quad z \in \Gamma', \quad |z| \geq A', 
\end{equation}
and
\begin{equation}\label{bjproperties}
|\partial^\alpha b_j (z)| \lesssim |a(z)|^{-1} \eabs{z}^{-2j-|\alpha|}, \quad \alpha \in \nn {2d}, \quad z \in \Gamma', \quad |z| \geq A', \quad j=0,1,2,\dots.   
\end{equation} 

Pick an integer $n > (m-m')/2$. We have  
\begin{equation*}
b = b_0 + \sum_{j=1}^{n-1} b_j + b_{n,r}
\end{equation*}
where $b_{n,r} \in G^{-m'-2n} \subseteq G^{-m}$. 

We obtain from \eqref{b0lowerbound} and \eqref{bjproperties} for some $C,C_1,C_2>0$
\begin{equation}\label{estdenom1}
\begin{aligned}
|b(z)| & \geq |b_0(z)| - \sum_{j=1}^{n-1} |b_j(z)| - |b_{n,r}(z)| \\
& \geq \ep |a(z)|^{-1} \left( 1 - C \sum_{j=1}^{n-1} \eabs{z}^{-2j} - C |a(z) |\eabs{z}^{-m'-2n} \right) \\
& \geq \ep |a(z)|^{-1} \left( 1 - C \sum_{j=1}^{n-1} \eabs{z}^{-2j} - C \eabs{z}^{m-m'-2n} \right) \\
& \geq C_1 |a(z)|^{-1} \\
& \geq C_2 \eabs{z}^{-m}, \quad z \in \Gamma', \quad |z| \geq A', 
\end{aligned}
\end{equation}
after possibly increasing $A'>0$. 

Secondly we have the estimate for any $\alpha \in \nn {2d}$, again using  \eqref{bjproperties}, 
\begin{equation}\label{estnom1}
\begin{aligned}
|\pd \alpha b(z)| & \leq \sum_{j=0}^{n-1} | \pd \alpha b_j(z)| + | \pd \alpha b_{n,r}(z)| \\
& \lesssim |a(z)|^{-1}  \eabs{z}^{-|\alpha|} \left( \sum_{j=0}^{n-1} \eabs{z}^{-2j} + |a(z) |\eabs{z}^{-m'-2n} \right) \\
& \lesssim |a(z)|^{-1}  \eabs{z}^{-|\alpha|} \left( \sum_{j=0}^{n-1} \eabs{z}^{-2j} + \eabs{z}^{m-m'-2n} \right) \\
& \lesssim |a(z)|^{-1}  \eabs{z}^{-|\alpha|}, \quad z \in \Gamma', \quad |z| \geq A'.
\end{aligned}
\end{equation}
Combining \eqref{estdenom1} and \eqref{estnom1} yields
\begin{equation*}
\left| \frac{\pd \alpha b(z)}{b(z)} \right| \lesssim \eabs{z}^{-|\alpha|}, \quad z \in \Gamma', \quad |z| \geq A', \quad \alpha \in \nn {2d}. 
\end{equation*}
This completes the verification of $\charac_{-m} (b) \cap \Gamma' = \emptyset$.
\end{proof}
\begin{lem}\label{lem:pdomod}
Let $u\in\cS'(\rr d)$, $s \in \ro$, $z_0 \in T^* \rr d \setminus 0$ and suppose $a \in G^0$ satisfies $z_0 \notin \charac(a)$ and $a^w(x,D)u\in Q^s(\rr{d})$. 
Then there exists an open conic set $\Gamma \subseteq T^* \rr d \setminus 0$ containing $z_0$ such that 
$b^w(x,D) u \in Q^s(\rr{d})$ for any $b \in G^0$ such that $\supp(b) \subseteq \Gamma$. 
\end{lem}
\begin{proof}
By Proposition \ref{microelliptic} there exists an open conic set $\Gamma' \subseteq T^* \rr d \setminus 0$ containing $z_0$ and $A>0$
such that if $\chi \in G^0$ satisfies $\supp(\chi) \subseteq \Gamma' \setminus B_A$ 
then there exists $c \in G^0$ such that 
$c \wpr a = \chi + r$ with $r \in \cS(\rr {2d})$. 

Let $\Gamma \subseteq T^* \rr d \setminus 0$ be an open conic set such that $z_0 \in \Gamma$ and $\overline{\Gamma \cap S_{2d-1}} \subseteq \Gamma'$. 
Suppose $\chi \in G^0$ satisfy $\supp (\chi) \subseteq \Gamma' \setminus B_A$ and $\chi(z) = 1$ for $z \in \Gamma$ and $|z| \geq 2A$, 
and let $b \in G^0$ satisfy $\supp(b) \subseteq \Gamma$. 
 
Then $b (1-\chi) \in C_c^\infty(\rr {2d})$, and $b \wpr \chi = b \chi + r$ with $r \in \cS(\rr {2d})$ since $b(z) \chi(z) = b(z)$ when $|z| \geq 2A$.  
Consequently we have, with $r_j \in \cS(\rr {2d})$, $j=1,2,3$, 
\begin{align*}
b^w(x,D) & = (b \chi)^w(X,D) + r_1^w(x,D) \\
& = b^w(x,D) \, \chi^w(X,D) + r_2^w(x,D) \\
& = b^w(x,D) \, c^w(X,D) \, a^w(x,D) + r_3^w(x,D). 
\end{align*}
The result now follows from $a^w(x,D) u \in Q^s(\rr d)$, the fact that operators of order zero are continuous on $Q^s(\rr d)$, 
and the fact that operators with Weyl symbols in $\cS(\rr {2d})$ are regularizing. 
\end{proof}

We have now reached a point where we can characterize the Sobolev--Gabor wave front set by means of pseudodifferential and localization operators. 

\begin{prop}\label{characterizationWFQs}
Let $u \in \cS'(\rr d)$ and $s \in \ro$. The following are equivalent:
\begin{enumerate}
\item $0 \neq z_0\notin \WF_{Q^s}(u)$;
\item There exists $b\in G^0$ such that $z_0 \notin \charac(b)$ and $b^w(x,D) u \in Q^s (\rr d)$; 
\item There exists $a\in G^s$ such that $z_0 \notin \charac(a)$ and $A_a u \in L^2(\rr d)$.
\end{enumerate}
\end{prop}

\begin{proof}
$(1)\Rightarrow (3)$: By definition
\begin{equation*}
\int_{\Gamma} \eabs{z}^{2s} | V_{\psi_0} u(z)|^2 \, d z < \infty
\end{equation*}
for some open conic set $\Gamma \subseteq T^* \rr d \setminus 0$ containing $z_0$. 
We may pick $a\in G^0$ supported in $\Gamma$ such that $z_0 \notin \charac (a)$ and conclude that 
the function $z \mapsto a(z)\eabs{z}^s V_{\psi_0} u(z)$ belongs to $L^2(\rr {2d})$. Since $V^*_{\psi_0}: L^2(\rr {2d}) \mapsto L^2(\rr d)$ is a continuous operator we consequently have
\begin{equation*}
(2\pi)^{-d} V^*_{\psi_0} \left( a \eabs{\cdot}^s V_{\psi_0} u\right) = A_{a \eabs{\cdot}^s} u \in L^2(\rr d).
\end{equation*}
Since multiplication by the elliptic symbol $\eabs{\cdot}^s$ is an isomorphism $G^0 \mapsto G^s$ 
and $z_0 \notin \charac(\eabs{\cdot}^s a)$ due to $z_0 \notin \charac(a)$,
we obtain (3).

$(3)\Rightarrow(2)$: By \eqref{eq:weylwick} we have $b^w(x,D)u=A_a u\in L^2(\rr{d})$ for $b=\pi^{-d} e^{-|\cdot|^2}*a \in G^s$
and  $z_0 \notin \charac(b)$. 
Claim (2) follows by applying $(\eabs{\cdot}^{-s})^w(x,D)$ to both sides.

$(2)\Rightarrow(1)$: Let $\fy \in \cS(\rr d) \setminus 0$. By Lemma \ref{lem:pdomod} we may assume that $b\in G^0$ satisfies $b(z) = 1$ when $z \in \Gamma$ and $|z| \geq 1$ where $\Gamma \subseteq T^* \rr d \setminus 0$ is an open conic neighborhood of $z_0$ and $b^w(x,D)u\in Q^s(\rr d)$. We write $1=b+r$, where $r\in G^0$ satisfies $r(z)=0$ for $z \in \Gamma$ and $|z| \geq 1$. Then 
\begin{equation*}
V_\fy u = V_\fy b^w(x,D) u + V_\fy r^w(x,D) u.
\end{equation*}
The assumption $b^w(x,D)u\in Q^s$ means
\begin{equation*}
\int_{\rr {2d}} \eabs{z}^{2s} | V_\varphi b^w(x,D) u (z)|^2 \, d z < \infty.
\end{equation*}
It thus suffices to show that $V_\psi r^w(x,D) u$, restricted to some open cone containing $z_0$, is of rapid decay. This follows from \cite[Corollary~3.9 and Remark~3.10]{Rodino1}.
\end{proof}

We are now in a position to compare the Sobolev--Gabor wave front set with the 
Gabor wave front set. 
We denote for a conical subset $\Gamma \subseteq T^* \rr d \setminus 0$ by $\overline \Gamma \subseteq T^* \rr d \setminus 0$ its closure in $T^* \rr d \setminus 0$. We have the following equality:
\begin{prop}\label{WFequality}
If $u\in\cS'(\rr{d})$ then 
\begin{equation*}
\WF_G(u) = \overline{\bigcup_{s \in \ro} \WF_{Q^s}(u)} \subseteq T^* \rr d \setminus 0.
\end{equation*}
\end{prop}
\begin{proof}
``$\supseteq$'': 
This is an immediate consequence of $\WF_{Q^s}(u) \subseteq \WF_G (u)$ for all $s \in \ro$, 
combined with the fact that $\WF_G(u) \subseteq T^* \rr d \setminus 0$ is a closed subset. 

``$\subseteq$'': We may assume $\overline{\bigcup_{s \in \ro} \WF_{Q^s}(u)} \neq T^* \rr d \setminus 0$. Let $0 \neq z_0\notin \overline{\bigcup_{s \in \ro} \WF_{Q^s}(u)}$. By assumption, since $\overline{\bigcup_{s \in \ro} \WF_{Q^s}(u)} \subseteq T^* \rr d \setminus 0$ is closed and conic, there exists an open conic set $\Gamma \subseteq T^* \rr d \setminus 0$ containing $z_0$ such that $\overline{\Gamma}\cap \overline{\bigcup_{s \in \ro} \WF_{Q^s}(u)}=\emptyset$. 
Pick $b\in G^0$ with $\supp(b) \subseteq \Gamma$ such that $z_0 \notin \charac(b)$. We claim $b^w(x,D)u\in \cS(\rr d)$. 

To prove this we fix $s\in\ro$ and note, again using Proposition \ref{characterizationWFQs} and Lemma \ref{lem:pdomod}, that for any $z \in \overline{\Gamma \cap S_{2d-1}}$ there exists an open cone $\Gamma_z \subseteq T^* \rr d \setminus 0$ and $a_z \in G^0$ such that $z \notin \charac(a_z)$, $z \in \supp (a_z) \subseteq \Gamma_z$ and $a_z^w(x,D)u\in Q^s(\rr d)$. 
By compactness we have a finite covering of the form
\begin{equation*}
\overline{\Gamma} \subseteq  \bigcup_{k=1}^n \Gamma_{k} 
\end{equation*}
where $\Gamma_{k} = \Gamma_{z_k}$ for some corresponding points $z_k \in \Gamma_k \cap S_{2d-1}$, $1 \leq k \leq n$. 
We may assume that these points satisfy $z_{j} \notin \Gamma_{k}$ for $j \neq k$. 

A partition of unity gives
\begin{equation*}
\sum_{k=1}^n  a_{k} (z) = 1, \quad z \in \overline{\Gamma}, \quad |z| \geq 1,  
\end{equation*}
for $a_{k} \in G^0$ such that $\supp (a_k) \subseteq \Gamma_k$, $z_{k} \notin \charac(a_k)$ 
(cf. \cite[Remark~2.5]{Rodino1})
and $a_{k}^w(x,D)u \in Q^s(\rr d)$ by Lemma \ref{lem:pdomod}, for $1 \leq k \leq n$. 

By \eqref{calculuscomposition1} we get
\begin{equation*}
b \wpr \sum_{k=1}^n a_k = b \sum_{k=1}^n a_k + r_1 = b + r_2
\end{equation*}
where $r_1, r_2 \in \cS(\rr {2d})$. 
Thus 
\begin{equation*}
b^w(x,D) u = \sum_{k=1}^n b^w(x,D) \, a_k^w(x,D) u - r_2^w(x,D) u \in Q^s(\rr d)
\end{equation*}
by virtue of $a_k^w(x,D) u \in Q^s (\rr d)$ for $1 \leq k \leq n$, the continuity on $Q^s (\rr d)$ of operators of order zero, and the continuity $r_2^w(x,D): \cS'(\rr d) \mapsto \cS(\rr d) \subseteq Q^s (\rr d)$. Since $s \in \ro$ is arbitrary we have $b^w(x,D) u \in \cS(\rr d)$, which shows that $z_0 \notin \WF_G(u)$. 
We have thus shown
\begin{equation*}
\WF_G(u) \subseteq \overline{\bigcup_{s \in \ro} \WF_{Q^s}(u)}.
\end{equation*}
\end{proof}

%%%%%%%%%%%%%%%%%%%%%%%%%%%%%%%%%%%
\section{Microlocal and micro-hypoelliptic properties}\label{sec:micro}
%%%%%%%%%%%%%%%%%%%%%%%%%%%%%%%%%%%

The Gabor wave front set is governed by the microlocal and microelliptic inclusions \eqref{microlocalelliptic1} with respect to pseudodifferential operators with Shubin symbols. 
We now prove analogous statements for the Sobolev--Gabor wave front set for pseudodifferential as well as localization operators.
\begin{prop}
\label{prop:microloc}
If $u \in \cS'(\rr d)$, $m,s \in \ro$ and $a\in G^m$
then 
\begin{equation*}
\WF_{Q^{s-m}}(a^w(x,D)u) \subseteq \WF_{Q^{s}}(u) \quad \mbox{and} \quad \WF_{Q^{s-m}}(A_a u) \subseteq \WF_{Q^s}(u).
\end{equation*}
\end{prop}

\begin{proof}
The proof is inspired by that of \cite[Theorem~5.4]{Cordes1}.
By Proposition \ref{prop:Sobolevprop} we have $u\in Q^{t}(\rr d)$ for some $t \in \ro$. 
Pick an integer $N \geq 1$ such that $t + 2 N \geq s$. 
Assume $0 \neq z_0 \notin \WF_{Q^s}(u)$.
By Proposition \ref{characterizationWFQs} there exists $b\in G^0$ such that $z_0 \notin \charac (b)$ and $b^w(x,D)u\in Q^s(\rr d)$. 
We have $(b^w(x,D))^N u\in Q^s$, by the $Q^s$-continuity of pseudodifferential operators of order zero, and $z_0 \notin \charac(b_N)$ holds, where $b_N = b \wpr \cdots \wpr b$ is the $N$-fold Weyl product of $N$ copies of $b$.

Next we study $(b^w(x,D))^N a^w(x,D) u$ which at a first glance can be said to belong to $Q^{t-m}(\rr d)$. 
We may write
\begin{multline*}
(b^w(x,D))^N a^w(x,D) u=(b^w(x,D))^{N-1} \underbrace{a^w(x,D)}_{\text{order }m}  \underbrace{b^w(x,D) u}_{\in Q^s}\\ + (b^w(x,D))^{N-1} [b^w(x,D),a^w(x,D)] u. 
\end{multline*}
It follows that 
\begin{equation*}
(b^w(x,D))^N a^w(x,D) u=(b^w(x,D))^{N-1} [b^w(x,D),a^w(x,D)] u + v
\end{equation*}
where $v \in Q^{s-m}$. The order of $[b^w(x,D),a^w(x,D)]$ is $m-2$. 
We have thus replaced $a^w(x,D)$ by an operator of lower order, and
\begin{equation*}
(b^w(x,D))^N a^w(x,D) u \in Q^{t-m+2} + Q^{s-m}. 
\end{equation*}

Repeating the argument by induction gives 
\begin{equation*}
(b^w(x,D))^N a^w(x,D) u \in Q^{t-m+2N} + Q^{s-m} \subseteq Q^{s-m},
\end{equation*}
since $t + 2 N \geq s$. 
By virtue of Proposition \ref{characterizationWFQs} this proves $z_0\notin \WF_{Q^{s-m}}(a^w(x,D)u)$, so we have shown 
\begin{equation*}
\WF_{Q^{s-m}}(a^w(x,D)u) \subseteq \WF_{Q^{s}}(u).
\end{equation*}
The statement about localization operators follows from Proposition \ref{WW}.
\end{proof}

Our final result is a reverse inclusion to Proposition \ref{prop:microloc}, a micro-hypoellipticity statement:
\begin{thm}
\label{thm:microell}
If $u \in \cS'(\rr d)$, $m,s \in \ro$ and $a \in G^m$ 
then 
\begin{equation*}
\WF_{Q^{s}}(u)\subseteq \bigcap_{m^\prime\leq m}\WF_{Q^{s-m^\prime}}(a^w(x,D)u) \cup \charac_{m^\prime} (a) 
\end{equation*}
and
\begin{equation*}
\WF_{Q^{s}}(u)\subseteq \bigcap_{m^\prime\leq m}\WF_{Q^{s-m^\prime}}(A_a u) \cup \charac_{m^\prime} (a). 
\end{equation*}
\end{thm}
\begin{proof}
Suppose $0 \neq z_0 \notin \WF_{Q^{s-m^\prime}}(a^w(x,D)u) \cup \charac_{m^\prime} (a)$ for some $m' \leq m$. 
By Proposition \ref{characterizationWFQs} there exists $c \in G^0$ with $z_0 \notin \charac (c)$ and $c^w(x,D) a^w(x,D) u \in Q^{s-m'}(\rr d)$. 
By an argument similar to the proof of $\charac_{m'} (b) \subseteq \charac_{m'} (a)$ in Proposition \ref{WW}, 
involving \eqref{calculuscomposition1}, 
it follows from $z_0 \notin \charac (c)$ and $z_0 \notin \charac_{m'} (a)$ that $z_0 \notin \charac_{m'} (c \wpr a)$.

Thus by Proposition \ref{microelliptic} there exist $\chi \in G^0$ with $z_0 \notin \charac (\chi)$, and $b \in G^{-m^\prime}$ with $z_0 \notin \charac_{-m} (b)$, such that 
\begin{equation*}
b \wpr c \wpr a = \chi + r
\end{equation*}
where $r \in \cS(\rr {2d})$. 
Hence 
\begin{equation*}
\chi^w(x,D) u = b^w(x,D) c^w(x,D) a^w(x,D) u - r^w(x,D) u. 
\end{equation*}
We have $r^w(x,D) u \in \cS \subseteq Q^s$, and $b^w(x,D) c^w(x,D) a^w(x,D) u \in Q^s$ since $c^w(x,D) a^w(x,D) u \in Q^{s-m'}$ and $b \in G^{-m'}$.

Hence $\chi^w(x,D) u \in Q^s(\rr d)$ so according to Proposition \ref{characterizationWFQs} we have shown that $z_0 \notin \WF_{Q^{s}}(u)$. This proves the first inclusion for pseudodifferential operators. 

To prove the second inclusion for localization operators,
let $z_0 \notin \WF_{Q^{s-m^\prime}}(A_au) \cup \charac_{m^\prime} (a)$ for some $m' \leq m$. 
Then 
\begin{equation*}
z_0\notin \WF_{Q^{s-m^\prime}}(b^w(x,D)u) \cup \charac_{m^\prime} (b)
\end{equation*}
where $b = \pi^{-d} a * e^{-|\cdot|^2}$ according to Proposition \ref{WW}.  
The second inclusion now follows from the first. 
\end{proof}
A special case of Theorem \ref{thm:microell} combined with Proposition \ref{prop:microloc} gives 
two inclusions that are Sobolev--Gabor wave front set versions of the inclusions  \eqref{microlocalelliptic1}. 

\begin{cor}
\label{cor:microinclu}
If $u \in \cS'(\rr d)$, $m,s \in \ro$ and $a \in G^m$
then 
\begin{equation*}
\WF_{Q^{s-m}}(a^w(x,D) u) \subseteq \WF_{Q^{s}}(u)\subseteq \WF_{Q^{s-m}}(a^w(x,D)u) \cup \charac (a) 
\end{equation*}
and
\begin{equation*}
\WF_{Q^{s-m}}(A_a u) \subseteq \WF_{Q^{s}}(u)\subseteq \WF_{Q^{s-m}}(A_au) \cup \charac (a). 
\end{equation*}
\end{cor}

We note that Proposition \ref{WFequality} and Corollary \ref{cor:microinclu} combined imply   
\eqref{microlocalelliptic1}. 

As a final consequence we obtain from Theorem \ref{thm:microell} with the same technique 
the following corollary, which says that hypoellipticity of a symbol implies micro-hypoellipticity with respect to the Gabor wave front set of the corresponding operator.
This is the Gabor wave front set version of \cite[Corollary 2.5.6]{Cordes1}.
\begin{cor}
\label{cor:hypoelliptic}
If $m \in \ro$, $a \in G^m$, $m'\leq m$ and $\charac_{m'}(a)=\emptyset$ then
\begin{equation*}
\WF_G(u)=\WF_G(a^w(x,D)u)=\WF_G(A_a u), \quad  u \in\cS'(\rr d).
\end{equation*}
\end{cor}
An example provides an interpretation of the previous results from a time-frequency point of view. 
\begin{example}
Let $m\geq 0$ and let $\Gamma_1, \Gamma_2 \subseteq T^* \rr d \setminus 0$ be two open cones such that $\Gamma_1\cup\Gamma_2 = T^* \rr d \setminus 0$
and $\Gamma_1 \setminus \Gamma_2$ and $\Gamma_2 \setminus \Gamma_1$ have non-empty interiors. Pick two positive symbols $\chi_1\in G^0$ and $\chi_2\in G^{-m}$ such that 
\begin{itemize}
\item $\chi_1(z)=1$ for $z\in\Gamma_1 \setminus \overline{\Gamma_2}$, 
\item $\chi_1(z)=0$ for $z \in \Gamma_2 \setminus (\Gamma_1 \cup B_1)$, 
\item $\charac_{-m}(\chi_1+\chi_2)=\emptyset$. 
\end{itemize}
\begin{center}
\begin{figure}[!ht]
\begin{tikzpicture}[
        scale=0.7
        ]
    
\draw[-] ({5 * cos(70.0)},{5 * sin(70.0)}) -- (0,0) -- ({5 * cos(20.0)},{5 * sin(20.0)}) ;
\fill[black,opacity=0.22] ({5 * cos(80.0)},{5 * sin(80.0)}) arc (80:370:5) -- (0,0) -- ({5 * cos(80.0)},{5 * sin(80.0)});
\fill[black,opacity=0.06,postaction={right color=lightgray, left color=black, opacity=0.2}] ({5 * cos(70.0)},{5 * sin(70.0)}) arc (70:80:5) -- (0,0) -- ({5 * cos(70.0)},{5 * sin(70.0)});
\fill[black,opacity=0.06,postaction={top color=lightgray, bottom color=black, opacity=0.2}] ({5 * cos(10.0)},{5 * sin(10.0)}) arc (10:20:5) -- (0,0) -- ({5 * cos(10.0)},{5 * sin(10.0)});

\node at (4.2,4.2) {$\Gamma_2$};
    \draw[->] (-4.5,0)--(4.5,0) node[below]{$x$};
    \draw[->] (0,-4.5)--(0,4.5) node[left]{$\xi$};

\draw[<->] ({5.6 * cos(70.0)},{5.6 * sin(70.0)}) arc (70:380:5.6); 
\draw[<->] ({5.4 * cos(10.0)},{5.4 * sin(10.0)}) arc (10:80:5.4); 
\node at (-3.2,5.4) {$\Gamma_1$};

\fill[white] (0,0) circle (0.1); 
\draw  (0,0) circle (0.1);

\clip ({5 * cos(20.0)},{5 * sin(20.0)}) -- (0,0) -- ({5 * cos(70.0)},{5 * sin(70.0)}) arc (70:20:5);
\draw[-] ({5 * cos(10.0)},{5 * sin(10.0)}) -- (0,0) -- ({5 * cos(80.0)},{5 * sin(80.0)}) ;
\fill [white, opacity=0, postaction={outer color=white, inner color=black, opacity=0.3}]  (0,0) circle (5);
\fill[white] (0,0) circle (0.1); 
\draw  (0,0) circle (0.1);
\end{tikzpicture}
\caption{A splitting of the time-frequency plane by a filter.}
\label{fig:appl}
\end{figure}
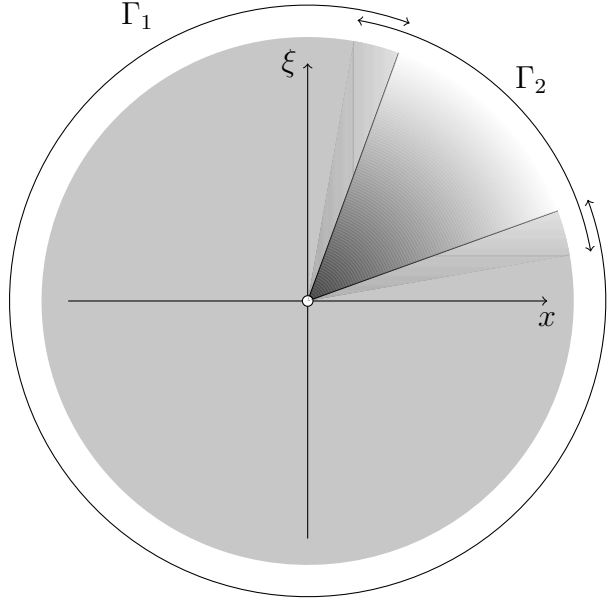
\end{center}

Figuratively speaking, $A:=A_{\chi_1+\chi_2}$ is the prototype of a filter, that when applied to a signal $u\in\cS'(\rr d)$ preserves the time-frequency contributions in $\Gamma_1\setminus \overline{\Gamma_2}$ and dampens the signal in $\Gamma_2\setminus\Gamma_1$. Such a splitting for $d=1$ is visualized in Figure \ref{fig:appl}. In the following, this will be interpreted from a microlocal point of view.

First of all, the total singularities are preserved $\WF_G(u) = \WF_G (A u)$ by Corollary \ref{cor:hypoelliptic}. Furthermore
\begin{equation*}
\WF_{Q^s}(Au)  \cap \Gamma_1\setminus \overline{\Gamma_2} = \WF_{Q^s}(u) \cap \Gamma_1\setminus \overline{\Gamma_2}
\end{equation*}
by Proposition \ref{prop:microloc} and Theorem \ref{thm:microell}, meaning that $A$ preserves the Sobolev--Gabor order of singularities in $\Gamma_1\setminus \overline{\Gamma_2}$. 

For $\Gamma_2\setminus\Gamma_1$ the order changes. We claim 
\begin{equation}\label{WFequalityfilter}
\WF_{Q^{s+m}}(Au) \cap \Gamma_2\setminus\Gamma_1 = \WF_{Q^s}(u) \cap \Gamma_2\setminus\Gamma_1. 
\end{equation}
Indeed, by Theorem \ref{thm:microell} we have $ \WF_{Q^s}(u) \subseteq \WF_{Q^{s+m}}(Au)$. 
If conversely $0 \neq z_0\notin \WF_{Q^s}(u)$ and $z_0 \in \Gamma_2 \setminus \Gamma_1$ then pick a pseudodifferential operator $B:=b^w(x,D)$ with symbol $b \in G^0$ supported in $\Gamma_2\setminus \Gamma_1$ such that $z_0 \notin \charac(b)$. Then $B A = B A_{\chi_2}+R$ where $R$ is regularizing, which follows from the asymptotic expansions \eqref{eq:weylwickexp} and \eqref{calculuscomposition1}, since the supports of $b$ and $\chi_1$ have compact intersection.

Thus $B A$ is a pseudodifferential operator of order $-m$ and $z_0 \notin \charac(b)$, and hence, since by Proposition \ref{prop:microloc} and Theorem \ref{thm:microell}, 
\begin{equation*}
\WF_{Q^{s+m}}(A u) \subseteq \WF_{Q^{s+m}}(BA u) \cup \charac(b)
\subseteq \WF_{Q^{s}}(u) \cup \charac(b)
\end{equation*}
we obtain $z_0\notin \WF_{Q^{s+m}}(Au)$. This proves \eqref{WFequalityfilter}.

Summing up we have shown that $A$ preserves the Sobolev--Gabor orders of singularities in $\Gamma_1\setminus \overline{\Gamma_2}$ and decreases the Sobolev--Gabor orders of singularities in $\Gamma_2\setminus\Gamma_1$ by $m$.
\end{example}

%%%%%%%%%%%%%%%%%%%%%%%%%%%%%

\end{document}